\documentclass{article}
\usepackage{authblk}
\usepackage{abstract}
\usepackage{amsmath}
\usepackage{wasysym}
\usepackage{amsthm}
\usepackage{amssymb}
\usepackage{enumerate}
\newtheorem{definition}{Definition}
\newtheorem{Proposition}{Proposition}
\newtheorem{Lemma}{Lemma}
\newtheorem{Theorem}{Theorem}

\begin{document}
	\title{Skew von Neumann constant in Weak Orlicz Spaces and Weak Lebesgue Spaces}
	\author[1]{Haoyu Zhou}
	\author[1,2]{Qi Liu\thanks{Qi Liu:liuq67@aqnu.edu.cn }}
	\author[1]{Yuxin Wang}
	\author[1]{Man Liang}
	\author[3]{Linlin Fu}
	\affil{School of Mathematics and physics, Anqing Normal University, Anqing 246133,P.R.China}
	\affil[2]{International Joint Research Center of Simulation and Control for Population Ecology of Yantze River in Anhui ,Anqing Normal University, Anqing 246133,P.R.China}
	\affil[3]{ School of Mathematics and Statistics, Shaoguan University, Shaoguan 512005, PR China}
	\maketitle 
	\begin{abstract}
	This paper defines the skew von Neumann constant in quasi-Banach spaces. Meanwhile, we obtain two constants. It presents the upper and lower bounds of two constants. Subsequently, it deduces the lower bound of the skew von Neumann constant within weak Orlicz spaces. Then, leveraging the relationship between weak Orlicz spaces and weak Lebesgue spaces, the lower bound of the skew von Neumann - Jordan constant in weak Lebesgue spaces is established. Meanwhile, in this paper, we also generalize the skew von Neumann constant. Then we present the lower bounds for the $p$-th von Neumann constant in weak Orlicz spaces and weak Lebesgue spaces.
	\end{abstract}
	\textbf{keywords:} {weak Orlicz spaces; weak Lebesgue spaces; quasi-Banach spaces; skew von Neumann constant}
	
	\section{Introduction} 
Let \( X \) be a Banach space. The von Neumann-Jordan constant associated with \( X \) (see \cite{2},\cite{3}) is defined as follows:

$$
C_{N J}(X):=\sup \left\{\frac{\|f-g\|_X^2+\|f+g\|_X^2}{2\left(\|f\|_X^2+\|g\|_X^2\right)}: f, g \in X \backslash\{0\}\right\} .
$$
This constant has long been favored by many scholars due to its favorable background significance and algebraic structure. For more results on the von Neumann-Jordan constant and the characterization of the geometric properties of Banach spaces via this constant, one can refer to reference \cite{18}-\cite{26}.

 Within the context of this study\cite{2} and \cite{3}. In particular, \( C_{N J}\left(L^2\left(\mathbb{R}^n\right)\right) = 1 \). For a general \( p \in [1, \infty] \), it is known that for finite \( p \), \( C_{N J}\left(L^p\left(\mathbb{R}^n\right)\right) = \max\left\{2^{\frac{2}{p}-1}, 2^{1-\frac{2}{p}}\right\} \), and \( C_{N J}\left(L^{\infty}\left(\mathbb{R}^n\right)\right) = 2 \).

The investigation into the von Neumann-Jordan constant of Lebesgue spaces can be generalized to Orlicz spaces. Let us recall the definition of these spaces (see \cite{4}). Let \( \Phi: [0, \infty) \rightarrow [0, \infty) \) be any \( N \)-function, that is, \( \Phi \) is convex, \( \Phi(0) = 0 \), \( \displaystyle\lim_{t \rightarrow 0} \frac{\Phi(t)}{t} = 0 \), and \( \displaystyle\lim_{t \rightarrow \infty} \frac{\phi(t)}{t} = \infty \). The Orlicz space \( L^{\Phi}\left(\mathbb{R}^n\right) \) is defined as the set of all measurable functions \( f \) on \( \mathbb{R}^n \) for which

$$
\|f\|_{L^{\Phi}}:=\inf \left\{\lambda>0: \int_{R^n} \Phi\left(\frac{|f(x)|}{\lambda}\right) d x \leq 1\right\}<\infty
$$

When \( \Phi(t) = t^p \) with \( 1 \leq p < \infty \), we have \( L^{\Phi}\left(\mathbb{R}^n\right) = L^p\left(\mathbb{R}^n\right) \). Some results concerning the von Neumann-Jordan constant of Orlicz spaces can be found in \cite{5}. One of the results in this book is

$$
C_{N J}\left(L^{\Phi}\left(\mathbb{R}^n\right)\right) \geq \max \left\{\frac{1}{\bar{\alpha}_{\Phi}}, 2 \bar{\beta}_{\Phi}^2\right\}
$$

where \( \bar{\alpha}_{\Phi} \) and \( \bar{\beta}_{\Phi} \) are defined by \( \bar{\alpha}_{\Phi}:=\displaystyle\inf _{t>0} \frac{\Phi^{-1}(t)}{\Phi^{-1}(2 t)} \) and \( \bar{\beta}_{\Phi}=\displaystyle\sup _{t>0} \frac{\Phi^{-1}(t)}{\Phi^{-1}(2 t)} \). 

As a scholarly contribution \cite{9}, this paper aims to examine the von Neumann-Jordan constant for weak Orlicz spaces and weak Lebesgue spaces. It should be recalled that the weak Orlicz space \( w L^{\Phi}\left(\mathbb{R}^n\right) \) is defined as the set of measurable functions \( f \) on \( \mathbb{R}^n \) such that

$$
\|f\|_{w L^{\Phi}\left(\mathbb{R}^n\right)}:=\inf \left\{b>0: \sup _{t>0} \Phi(t)\left|\left\{x \in \mathbb{R}^n: \frac{|f(x)|}{b}>t\right\}\right| \leq 1\right\}<\infty
.$$

It is worth noting that \( L^{\Phi}\left(\mathbb{R}^n\right) \subseteq w L^{\Phi}\left(\mathbb{R}^n\right) \) and \( \|f\|_{w L^{\Phi}} = \displaystyle\sup _{t>0}\left\|t \chi_{\{|f|>t\}}\right\|_{L^{\Phi}} \). Meanwhile, the weak Lebesgue space \( w L^p\left(\mathbb{R}^n\right) \) (for \( 1 <p < \infty \)) is defined as the collection of measurable functions \( f \) on \( \mathbb{R}^n \) for which

$$
\|f\|_{w L^p}=\sup _{\gamma>0} \gamma\left(\left\{x \in \mathbb{R}^n:|f(x)|>\gamma\right\}\right)^{\frac{1}{p}}<\infty .
$$

If \( \Phi(t) := t^p \) for some \( 1 \leq p < \infty \), then \( w L_{\Phi}\left(\mathbb{R}^n\right) = w L^p\left(\mathbb{R}^n\right) \). Thus, \( L_{\Phi}\left(\mathbb{R}^n\right) \) can be regarded as a generalization of the weak Lebesgue space \( L^p\left(\mathbb{R}^n\right) \). Note that

$$
\|f\|_{w L_{\Phi}\left(\mathbb{R}^n\right)}:=\sup _{t>0}\left\|t \chi_{\{|f|>t\}}\right\|_{L_{\Phi}\left(\mathbb{R}^n\right)}.
$$
 Within the context of this study \cite{27}, a quasi-normed space $X$ belongs to the class of locally bounded topological vector spaces. Equivalently, the topology on $X$ arises from a quasi-norm — a mapping \(\|\cdot\|: X \to [0, \infty)\) satisfying the following conditions:

(i) \(\|x\| = 0\) holds if and only if \(x = 0\);

(ii) For any scalar \(\alpha \in \mathbb{R}\) and vector \(x \in X\), \(\|\alpha x\| = |\alpha| \|x\|\);

(iii) There exists a constant \(C \geq 1\) such that for all \(x, y \in X\), the inequality$$\|x + y\| \leq C \left( \|x\| + \|y\| \right) $$
holds.

When \(C = 1\), the quasi-norm reduces to a standard norm. A quasi-norm \(\|\cdot\|\) is termed a $p$-norm (with \(0 < p < 1\)) if it satisfies p-subadditivity, meaning:$$\|x + y\|^p \leq \|x\|^p + \|y\|^p \quad \text{for all } x, y \in X.$$

For more results on quasi-Banach spaces, please refer to the references \cite{29}-\cite{33}.

 H. Gunawan and Ifronika, in their literature \cite{9}, for a quasi - Banach space $X$, provide the definition of the von Neumann - Jordan constant $C_{NJ}(X)$ as
 $$C_{NJ}(X):=\sup\left\{\frac{||f+g||_X^2+||f-g||_X^2}{2C_X^2(||f||_X^2+||g||_X^2)}:f,g\in X,\text{ not both zero}\right\}$$
 where \(C_X = \sup\left\{ \frac{\|f + g\|_X}{\|f\|_X + \|g\|_X} : f,g \in X,(f,g) \neq (0,0) \right\}\). They also extend the research to Orlicz spaces, give lower-bound estimates of the von Neumann-Jordan constants for weak Orlicz spaces and weak Lebesgue spaces, and prove that the von Neumann-Jordan constant for the weak Lebesgue space \(wL^p\) tends to 2 as \(p \to \infty\). (The proof relies on the refined improvement of the positive constant in the triangle inequality of the weak Lebesgue space \( wL^p \)), thus supplementing new achievements to the research on the geometric properties of relevant spaces.
 
 For more knowledge about weak Orlicz spaces and weak Lebesgue spaces, please refer to the references \cite{10}-\cite{17}.

In a novel research endeavor, Liu et al \cite{6} defined a skewed version of the von Neumann constant. For $\xi,\eta>0$, this constant is denoted as $L_{YJ}(\xi,\eta,X)$ and is defined in the following manner:
	$$L_{YJ}(\xi,\eta,X)=\sup\left\{\frac{\|\xi x+\eta y\|^2+\|\eta x-\xi y\|^2}{(\xi^2+\eta^2)(\|x\|^2+\|y\|^2)}:x,y\in X, (x,y)\neq(0,0)\right\}.$$
Besides, the authors of \cite{7} present a comparable constant:
$$L^{\prime}_{YJ}(\xi,\eta,X)=\sup\left\{\frac{\|\xi x+\eta y\|^2+\|\eta x-\xi y\|^2}{2(\|x\|^2+\|y\|^2)}:x,y\in S_X\right\}.$$
In  recent years, several scholars have initiated investigations into the $L_{YJ}(\xi,\eta,X)$ constant. Furthermore, in the same year, the authors of \cite{8} derived the precise values of the $L_{YJ}(\xi,\eta,X)$ constant and the $L^{\prime}_{YJ}(\xi,\eta,X)$ constant for the regular octagon space. The authors of Reference \cite{28} further investigated this constant and presented an inequality relationship between the constant \( L_{YJ}(\lambda, \mu, X) \) and the generalized James constant.
Based on these two constants, some new constants have also been defined, and specific details can be found in references \cite{34}-\cite{36}.

	\section{Main results}
	Since the constant has parameters, different from reference \cite{9}, we need to introduce $C_{X_1}$ and $C_{X_2}$, which are related to the skew parameters in the definition of the quasi-norm. $C_{X_1}$ and $C_{X_2}$ are defined as:
	$$C_{X_1}=\sup\left\{\frac{\|\lambda f+\mu g\|_X}{\lambda\|f\|_X+\mu\|g\|_X}:f,g\in X, (f,g)\neq(0,0)\right\}$$
	and
	$$C_{X_2}=\sup\left\{\frac{\|\mu f-\lambda g\|_X}{\mu\|f\|_X+\lambda\|g\|_X}:f,g\in X, (f,g)\neq(0,0)\right\}.$$

	\begin{definition}
		For a quasi-Banach space $X$, and for $\lambda,\mu>0$, the skew von Neumann-Jordan constant $L_{YJ}^C(\lambda,\mu,X)$ is defined by 
		$$L_{YJ}^{C}(\lambda,\mu,X)=\sup\left\{\frac{\|\lambda f+\mu g\|_X^{2}+\|\mu f-\lambda g\|_X^{2}}{C_{X_1}C_{X_2}(\lambda^2+\mu^2)(\|f\|_X^{2}+\|g\|_X^{2})}:f,g\in X, (f,g)\neq(0,0)\right\}.$$
		
	\end{definition}

	After in-depth analysis of the norm structure of the weak Lebesgue space $wL^{\Phi}(\mathbb{R}^n)$, based on the exisiting  research foundation, we present the following key Lemmas, which is an important cornerstone for the subsequent derivation of the main theorem.
\begin{Lemma}\label{t1}
	\cite{1}Let $E$ be a measurable set of $\mathbb{R}^n$ and $0<|E|<\infty$. Then 
	$$\|\chi_E\|_{wL^{\Phi}(\mathbb{R}^n)}=\frac{1}{\Phi^{-1}\left(\frac{1}{|E|}\right)}.$$
\end{Lemma}
\begin{Lemma}\label{t7}
	Let $1< p<\infty$ and define
	$$C_{p_1}=\sup_{(f,g)\neq(0,0)}\frac{\|\lambda f+\mu g\|_{wL^p}}{\lambda\|f\|_{wL^p}+\mu\|g\|_{wL^p}} $$and
	$$C_{p_2}=\sup_{(f,g)\neq(0,0)}\frac{\|\mu f-\lambda g\|_{wL^p}}{\mu\|f\|_{wL^p}+\lambda\|g\|_{wL^p}}.$$ 
	Then$$(\lambda+\mu)^{\frac{1}{p}}\leq C_{p_1}\leq2 $$ and $$\frac{|\mu-\lambda|^{\frac{1}{p}+1}}{\lambda+\mu}\leq C_{p_2}\leq2 .$$
\end{Lemma}
\begin{proof}
	It is observed that, for every $t>0$, we have
	$$\begin{aligned}|\{x\in\mathbb{R}^{n}:|\lambda f(x)+\mu g(x)|>t\}|&\leq|\{x\in\mathbb{R}^{n}:|\lambda f(x)|>\frac{t}{2}\}|+|\{x\in\mathbb{R}^{n}:|\mu g(x)|>\frac{t}{2}\}|\\&=\left(\frac{t}{2}\right)^{-p}\|\lambda f\|_{wL^{p}}^{p}+\left(\frac{t}{2}\right)^{-p}\|\mu g\|_{wL^{p}}^{p}\\&=t^{-p}2^{p}(\|\lambda f\|_{wL^{p}}^{p}+\|\mu g\|_{wL^{p}}^{p}).\end{aligned}$$
	After multiplying by $t^p$ and computing the $p$-th root, we get
	$$\begin{aligned}t|\{x\in\mathbb{R}^{n}:|\lambda f(x)+\mu g(x)|>t\}|^{1/p}&\leq2(\|\lambda f\|_{wL^{p}}^{p}+\|\mu g\|_{wL^{p}}^{p})^{1/p}\\&\leq2(\lambda\|f\|_{wL^{p}}+\mu\|g\|_{wL^{p}}).\end{aligned}$$
	By taking the supremum for all $t>0$, we obtain
	$$\|\lambda f+\mu g\|_{wL^p}\leq2(\lambda\|f\|_{wL^p}+\mu\|g\|_{wL^p}).$$
	Let $f_1(x)=x^{-\frac{1}{P}}\chi_{(0,1)}$ and $g_2(x)=(1-x)^{-\frac{1}{p}}\chi_{(0,1)}$. Then
	$$\|f_1\|_{wL_{p}}=\|g_1\|_{wL^{p}}=1.$$
	Let $h_1(x)=\lambda f_1(x)+\mu g_1(x)$. As a result, for every $a\in(0,\frac{1}{2}]$, it follows that
	$$\begin{aligned}h_1(a)|\{x:|h_1(x)|>h_1(a)\}|^{\frac{1}{p}}&=\left(\lambda a^{-\frac{1}{p}}+\mu (1-a)^{-\frac{1}{p}}\right)((\lambda+\mu)a)^{\frac{1}{p}}\\&=(\lambda+\mu)^{\frac{1}{p}}\left(\lambda+\mu\left(\frac{a}{1-a}\right)^{\frac{1}{p}}\right).\end{aligned}$$
	By taking the supremum for $a\in(0,\frac{1}{2}]$, we obtain $$\|\lambda f_1+\mu g_1\|_{wL^{p}}=(\lambda+\mu)^{\frac{1}{p}}\cdot(\lambda+\mu)=(\lambda+\mu)^{1+\frac{1}{p}}.$$ This suggests 
	$$C_{p_1}\geq\frac{\|\lambda f_1+\mu g_1\|_{wL^p}}{\lambda\|f_1\|_{wL^p}+\mu\|g_1\|_{wL^p}}=\frac{(\lambda+\mu)^{1+\frac{1}{p}}}{\lambda+\mu}=(\lambda+\mu)^{\frac{1}{p}}.$$
	Therefore, we have completed the proof of the first half. Next, we will present the proof of the second half.
	
	It is observed that, for every $t>0$, we have
	$$\begin{aligned}|\{x\in\mathbb{R}^{n}:|\mu f(x)-\lambda g(x)|>t\}|&\leq|\{x\in\mathbb{R}^{n}:|\mu f(x)|>\frac{t}{2}\}|+|\{x\in\mathbb{R}^{n}:|\lambda g(x)|>\frac{t}{2}\}|\\&=\left(\frac{t}{2}\right)^{-p}\|\mu f\|_{wL^{p}}^{p}+\left(\frac{t}{2}\right)^{-p}\|\lambda g\|_{wL^{p}}^{p}\\&=t^{-p}2^{p}(\|\mu f\|_{wL^{p}}^{p}+\|\lambda g\|_{wL^{p}}^{p}).\end{aligned}$$
	After multiplying by $t^p$ and computing the $p$-th root, we get
	$$\begin{aligned}t|\{x\in\mathbb{R}^{n}:|\mu f(x)-\lambda g(x)|>t\}|^{1/p}&\leq2(\|\mu f\|_{wL^{p}}^{p}+\|\lambda g\|_{wL^{p}}^{p})^{1/p}\\&\leq2(\mu\|f\|_{wL^{p}}+\lambda\|g\|_{wL^{p}}).\end{aligned}$$
	By taking the supremum for all $t>0$, we obtain
	$$\|\mu f+\lambda g\|_{wL^p}\leq2(\mu\|f\|_{wL^p}-\lambda\|g\|_{wL^p}).$$
	Let $f_2(x)=x^{-\frac{1}{P}}\chi_{(0,1)}$ and $g_2(x)=(1-x)^{-\frac{1}{p}}\chi_{(0,1)}$. Then
	$$\|f_2\|_{wL_{p}}=\|g_2\|_{wL^{p}}=1.$$
	Let $h_2(x)=\mu f_2(x)-\lambda g_2(x)$. As a result, for every $a\in(0,\frac{1}{2}]$, it follows that
	$$\begin{aligned}h_2(a)|\{x:|h_2(x)|>h_2(a)\}|^{\frac{1}{p}}&=\left|\mu a^{-\frac{1}{p}}-\lambda (1-a)^{-\frac{1}{p}}\right|(|\mu-\lambda|a)^{\frac{1}{p}}\\&=|\mu-\lambda|^{\frac{1}{p}}\left|\mu-\lambda\left(\frac{a}{1-a}\right)^{\frac{1}{p}}\right|.\end{aligned}$$
	By taking the supremum for $a\in(0,\frac{1}{2}]$, we obtain $$\|\mu f_2-\lambda g_2\|_{wL^{p}}=|\mu-\lambda|^{\frac{1}{p}}\cdot|\mu-\lambda|=|\mu-\lambda|^{1+\frac{1}{p}}.$$ This suggests 
	$$C_{p_2}\geq\frac{\|\lambda f_2-\mu g_2\|_{wL^p}}{\lambda\|f_2\|_{wL^p}+\mu\|g_2\|_{wL^p}}=\frac{|\mu-\lambda|^{1+\frac{1}{p}}}{\lambda+\mu}.$$
	This proof is completed.

\end{proof}

According to reference \cite{9}, we can present the following Lemma \ref{t5} and Lemma \ref{t6}.

\begin{Lemma}\label{t5}
	Let $1<p<\infty$ and define
	$\|f\|_{wL^p}^*:=\displaystyle\sup_{|E|\subset\mathbb{R}^n}|E|^{\frac{1}{P}-1}\left(\int_{R^n}|f(x)|dx\right).$
	Then we have $\|f\|_{wL^p}^*\leq\frac{p}{p-1}\|f\|_{wL^p}.$
\end{Lemma}

\begin{Lemma}\label{t6}
	Let $1<p<\infty$. Then we have
	$\|f\|_{wL^p}^*\leq\frac{p}{p-1}\|f\|_{wL^p}.$
\end{Lemma}
Taking the preceding lemmas as our basis, we will now undertake the tasknof finishing the proof outlined below.
\begin{Proposition}
		Let $1< p<\infty$ and define
	$$C_{p_1}=\sup_{(f,g)\neq(0,0)}\frac{\|\lambda f+\mu g\|_{wL^p}}{\lambda\|f\|_{wL^p}+\mu\|g\|_{wL^p}} $$and
	$$C_{p_2}=\sup_{(f,g)\neq(0,0)}\frac{\|\mu f-\lambda g\|_{wL^p}}{\mu\|f\|_{wL^p}+\lambda\|g\|_{wL^p}}.$$ 
	Then$$(\lambda+\mu)^{\frac{1}{p}}\leq C_{p_1}\leq\min\left\{2,\frac{p}{p-1}\right\} $$ and $$\frac{|\mu-\lambda|^{\frac{1}{p}+1}}{\lambda+\mu}\leq C_{p_2}\leq\min\left\{2,\frac{p}{p-1}\right\}  .$$
\end{Proposition}
\begin{proof}
	On the basis of Lemma \ref{t5} and Lemma \ref{t6}. The normability inherent to the weak $L^p$, we get
	$$\begin{aligned}\|\lambda f+\mu g\|_{wL^p}&\leq\|\lambda f+\mu g\|_{wL^{p}}^{*}\\&\leq\left(\frac{p}{p-1}\right)\lambda\|f\|_{wL^{p}}+\left(\frac{p}{p-1}\right)\mu\|g\|_{wL^{p}}\\&=\left(\frac{p}{p-1}\right)\left(\lambda\|f\|_{wL^{p}}+\mu\|g\|_{wL^{p}}\right).\end{aligned}$$
	Similarly, we can obtain $\|\mu f-\lambda g\|_{wL^p}\leq\left(\frac{p}{p-1}\right)\left(\mu\|f\|_{wL^{p}}+\lambda\|g\|_{wL^{p}}\right)$.
	Combining with Lemma \ref{t7}, we can obtain
	$$C_{p_1}=\sup_{(f,g)\neq(0,0)}\frac{\|\lambda f+\mu g\|_{wL^p}}{\lambda\|f\|_{wL^p}+\mu\|g\|_{wL^p}}\leq\min\left\{2,\frac{p}{p-1}\right\}$$
	and
	$$C_{p_2}=\sup_{(f,g)\neq(0,0)}\frac{\|\mu f-\lambda g\|_{wL^p}}{\mu\|f\|_{wL^p}+\lambda\|g\|_{wL^p}}\leq\min\left\{2,\frac{p}{p-1}\right\}.$$

\end{proof}

	\begin{Theorem}\label{t2}
		Let $\Phi$ be any N-function, then$$L_{YJ}^C(wL^{\Phi}(\mathbb{R}^{n}))\geq\max\left\{\frac{1}{C_{\Phi_1}C_{\Phi_2}\bar{\alpha}_{\Phi}^{2}},\frac{2\bar{\beta}_{\Phi}^{2}}{C_{\Phi_1}C_{\Phi_2}}\right\}.$$
	\end{Theorem}
	\begin{proof}
		By the definition of $\bar{\alpha}_\Phi$, for any $\varepsilon>0$, there exists some $t_0>0$ such that
		$$\frac{\Phi^{-1}(t_0)}{\Phi^{-1}(2t_0)}<\bar{\alpha}_\Phi+\varepsilon.$$
		First, define $r_0$ as $r_0:=(2t_0|B(0,1)|)^{-\frac{1}{n}}$. Next, select two points $x_1,x_2\in\mathbb{R}^n$ so that the balls $B(x_1,r_0)$ and $B(x_2,r_0)$ are disjoint. For each index $i=1,2$, define the function $f_i$ as $f_i:=\Phi^{-1}(2t_0)\chi_{B(x_i,r_0)}$. After that, by applying Lemma \ref{t1}, the following result holds		
		$$\begin{aligned}
	\|f_i\|_{wL^\Phi}&=\Phi^{-1}(2t_0)\|B(x_i,r_0)\|_{wL^\Phi}\\&=\frac{\Phi^{-1}(2t_0)}{\Phi^{-1}(\frac{1}{|B(x_i,r_0)|})}\\&=1.	\end{aligned}$$
		Given the disjointness of $B(x_1,r_0)$ and $B(x_2,r_0)$, it follows that
		$$\begin{aligned}\|\lambda f_{1}+\mu f_{2}\|_{wL^{\Phi}}&=\Phi^{-1}(2t_{0})\|\lambda\chi_{B(x_{1},r_{0})}+\mu\chi_{B(x_{2},r_{0})}\|_{wL^{\Phi}}\\&=(\lambda+\mu)\Phi^{-1}(2t_0)\|\chi_{B(x_1,r_0)\cup B(x_2,r_0)}\|_{wL^\Phi}\\&=\frac{(\lambda+\mu)\Phi^{-1}(2t_0)}{\Phi^{-1}(t_0)}\\&>\frac{\lambda+\mu}{\bar{\alpha}_{\Phi}+\varepsilon}.\end{aligned}$$
		Similarly,
		$$\|\mu f_1-\lambda f_2\|_{wL^{\Phi}}>\frac{|\mu-\lambda|}{\bar{\alpha}_{\Phi}+\varepsilon}.$$
		Based on  what has been obtained above, we have
		$$\begin{aligned}
	L_{YJ}^C(wL^\Phi(\mathbb{R}^n))&\geq\frac{\|\lambda f_1+\mu f_2\|_{wL^\Phi}^2+\|\mu f_1-\lambda f_2\|_{wL^\Phi}^2}{C_{\Phi_1}C_{\Phi_2}(\lambda^2+\mu^2)(\|f_1\|_{wL^\Phi}^2+\|f_2\|_{wL^\Phi}^2)}\\&\geq\frac{1}{C_{\Phi_1}C_{\Phi_2}(\bar{\alpha}_\Phi+\epsilon)^2}.\end{aligned}$$
		We can conclude that
		$$L_{YJ}^C(wL^{\Phi}(\mathbb{R}^{n}))\geq\frac{1}{C_{\Phi_1}C_{\Phi_2}\bar{\alpha}_{\Phi}^{2}}.$$
		Similarly, from the definition of $\bar{\beta}_\Phi$, for arbitrary $\epsilon>0$, there exists $u_0>0$ where
		$$\frac{\Phi^{-1}(u_0)}{\Phi^{-1}(2u_0)}>\bar{\beta}_\Phi-\frac{\varepsilon}{2}.$$
		Next, introduce $v_0$ by setting $v_0:=(2u_0|B(0,1)|)^{-\frac{1}{n}}$. Next, select two points $y_1,y_2\in\mathbb{R}^n$ so that the balls $B(y_1,r_0)$ and $B(y_2,r_0)$ are disjoint. Then, define two functions: let $g_1$ be given by $g_1:=\Phi^{-1}(\chi_{B(y_1,v_0)}+\chi_{B(y_2,v_0)})$, and $g_2$ be given by $g_2:=\Phi^{-1}(\chi_{B(y_1,v_0)}-\chi_{B(y_2,v_0)})$.
		
		By Lemma \ref{t1}, we thus observe that
		$$\begin{aligned}\|g_{1}\|_{wL^{\phi}}&=\Phi^{-1}(u_{0})\|\chi_{B(y_{1},v_{0})\cup B(y_{2},v_{0})}\|_{wL^{\Phi}}\\&=\frac{\Phi^{-1}(u_{0})}{\Phi^{-1}\left(\frac{1}{2|B(y_{1},v_{0})|}\right)}\\&=1.\end{aligned}$$
		Similarly, we have $\|g_{2}\|_{wL^{\phi}}=1$.
		
			Given the disjointness of $B(x_1,r_0)$ and $B(x_2,r_0)$, it follows that
		$$\begin{aligned}
		\|\lambda g_{1}+\mu g_{2}\|_{wL^{\Phi}}&=2\lambda\Phi^{-1}(u_{0})\|\chi_{B(y_{1},v_{0})}\|_{wL^{\Phi}}\\&=\frac{2\lambda\Phi^{-1}(u_{0})}{\Phi^{-1}(2u_{0})}\\&>\lambda(2\bar{\beta}-\varepsilon)\end{aligned}$$
		and 
		$$\begin{aligned}
	\|\mu g_{1}-\lambda g_{2}\|_{wL^{\Phi}}&=2\mu\Phi^{-1}(u_{0})\|\chi_{B(y_{1},v_{0})}\|_{wL^{\Phi}}\\&=\frac{2\mu\Phi^{-1}(u_{0})}{\Phi^{-1}(2u_{0})}\\&>\mu(2\bar{\beta}-\varepsilon).\end{aligned}$$
			Based on  what has been obtained above, we have
	$$\begin{aligned}
L_{YJ}^C(wL^{\Phi}(\mathbb{R}^{n}))&\geq\frac{\|\lambda g_{1}+\mu g_{2}\|_{wL^{\Phi}}^{2}+\|\mu g_{1}-\lambda g_{2}\|_{wL^{\Phi}}^{2}}{C_{\Phi_1}C_{\Phi_2}(\lambda^2+\mu^2)(\|g_{1}\|_{wL^{\Phi}}^{2}+|\|g_{2}\|_{wL^{\Phi}}^{2})}\\&\geq\frac{(2\bar{\beta}_{\Phi}-\varepsilon)^{2}}{2C_{\Phi_1}C_{\Phi_2}}.\end{aligned}$$
		We can conclude that
		 $$L_{YJ}^C(wL^\Phi(\mathbb{R}^n))\geq\frac{2\bar{\beta}_\Phi^2}{C_{\Phi_1}C_{\Phi_2}}.$$
	\end{proof}
	As a consequence of Theorem \ref{t2}, we can derive a lower bound for the skew von Neumann-Jordan constant of weak Lebesgue spaces.
		\begin{Theorem}\label{t4}
			If $p>1$, then
				$$L_{YJ}^{C}(wL^{\Phi}(\mathbb{R}^{n}))\geq\max\left\{\frac{2^{\frac{2}{p}}}{C_{p_1}C_{p_2}},\frac{2^{1-\frac{2}{p}}}{C_{p_1}C_{p_2}}\right\}.$$
		
		\end{Theorem}
	\begin{proof}
		We commence by demonstrating the lower bound within inequality. Define the function $\Phi(u)$ as $\Phi(u)=u^p$. In this set up,
		$$\bar{\alpha}_{\Phi}=\inf_{u>0}\frac{\Phi^{-1}(u)}{\Phi^{-1}(2u)}=\inf_{u>0}\frac{u^{\frac{1}{p}}}{(2u)^{\frac{1}{p}}}=2^{-\frac{1}{p}}$$
		and 
		$$\bar{\beta}_{\Phi}=\sup_{u>0}\frac{\Phi^{-1}(u)}{\Phi^{-1}(2u)}=\sup_{u>0}\frac{u^{\frac{1}{p}}}{(2u)^{\frac{1}{p}}}=2^{-\frac{1}{p}}.$$
		In the case where $\Phi=u^p$, it follows that $wL^\Phi(\mathbb{R}^n)=wL^p(\mathbb{R}^n)$. According to Theorem \ref{t2}, we then obtain
		$$\max\left\{\frac{2^{\frac{2}{p}}}{C_{p_1}C_{p_2}},\frac{2^{1-\frac{2}{p}}}{C_{p_1}C_{p_2}}\right\}\leq L_{YJ}^C(wL^{\Phi}(\mathbb{R}^{n})).$$

\end{proof}
	
\section{The constant $L_{YJ}^{C,p}(\lambda,\mu,X)$ }	
	Regarding the above - mentioned constant \( L_{YJ}^{C}(\lambda,\mu,X) \), we generalize the constant \( L_{YJ}^{C}(\lambda,\mu,X) \), and this generalization encompasses some previous conclusions. Next, we will formally present the definition of the constant \( L_{YJ}^{C,p}(\lambda,\mu,X) \) to compute the results in weak Orlicz spaces and weak Lebesgue spaces.
	\begin{definition}
			For a quasi-Banach space $X$, and for $\lambda,\mu>0$, the skew von Neumann-Jordan constant $L_{YJ}^{C,p}(\lambda,\mu,X)$ is defined by 
		$$L_{YJ}^{C,p}(\lambda,\mu,X)=\sup\left\{\frac{\|\lambda f+\mu g\|^p+\|\mu f-\lambda g\|^p}{C_{X_1}C_{X_2}(\lambda^p+\mu^p)(\|f\|^p+\|g\|^p)}:f,g\in X, (f,g)\neq(0,0)\right\}$$
			where
		$$C_{X_1}=\sup\left\{\frac{\|\lambda f+\mu g\|_X}{\lambda\|f\|_X+\mu\|g\|_X}:f,g\in X, (f,g)\neq(0,0)\right\}$$
		and
		$$C_{X_2}=\sup\left\{\frac{\|\mu f-\lambda g\|_X}{\mu\|f\|_X+\lambda\|g\|_X}:f,g\in X, (f,g)\neq(0,0)\right\}.$$
	\end{definition}
	
	\begin{Theorem}\label{t3}
		Let $\Phi$ be any N-function, then$$L_{YJ}^{C,p}(wL^{\Phi}(\mathbb{R}^{n}))\geq\max\left\{\frac{(\lambda+\mu)^p+|\mu-\lambda|^p}{2C_{\Phi_1}C_{\Phi_2}\bar{\alpha}_{\Phi}^{p}(\lambda^p+\mu^p)},\frac{2^{p-1}\bar{\beta}_\Phi^p}{C_{\Phi_1}C_{\Phi_2}}\right\} .$$
	\end{Theorem}
	\begin{proof}
		According to Theorem \ref{t2}, we can obtain
		$\|f_1\|_{wL^\Phi}=\|f_2\|_{wL^\Phi}=1$, $$\|\lambda f_{1}+\mu f_{2}\|_{wL^{\Phi}}>\frac{\lambda+\mu}{\bar{\alpha}_{\Phi}+\varepsilon}$$ and $$\|\mu f_1-\lambda f_2\|_{wL^{\Phi}}>\frac{|\mu-\lambda|}{\bar{\alpha}_{\Phi}+\varepsilon}.$$
		
		Thus, we have
		$$\begin{aligned}
	L_{YJ}^{C,p}(wL^\Phi(\mathbb{R}^n))&\geq\frac{\|\lambda f_1+\mu f_2\|_{wL^\Phi}^p+\|\mu f_1-\lambda f_2\|_{wL^\Phi}^p}{C_{\Phi_1}C_{\Phi_2}(\lambda^p+\mu^p)(\|f_1\|_{wL^\Phi}^p+\|f_2\|_{wL^\Phi}^p)}\\&\geq\frac{(\lambda+\mu)^p+|\mu-\lambda|^p}{2C_{\Phi_1}C_{\Phi_2}(\lambda^p+\mu^p)(\bar{\alpha}_\Phi+\epsilon)^p}.	\end{aligned}$$
		
		We can conclude that
		$$L_{YJ}^{C,p}(wL^{\Phi}(\mathbb{R}^{n}))\geq\frac{(\lambda+\mu)^p+|\mu-\lambda|^p}{2\bar{\alpha}_{\Phi}^{p}(\lambda^p+\mu^p)}.$$
		
		Similarly, based on Theorem \ref{t2}, we can also obtain $\|g_{1}\|_{wL^{\phi}}=\|g_{2}\|_{wL^{\phi}}=1$,
		$$\|\lambda g_{1}+\mu g_{2}\|_{wL^{\Phi}}>\lambda(2\bar{\beta}-\varepsilon)$$ and 
		$$\|\mu g_{1}-\lambda g_{2}\|_{wL^{\Phi}}>\mu(2\bar{\beta}-\varepsilon).$$
		
		Thus, we have
		$$\begin{aligned}
	L_{YJ}^{C,p}(wL^{\Phi}(\mathbb{R}^{n}))&\geq\frac{\|\lambda g_{1}+\mu g_{2}\|_{wL^{\Phi}}^{p}+\|\mu g_{1}-\lambda g_{2}\|_{wL^{\Phi}}^{p}}{C_{\Phi_1}C_{\Phi_2}(\lambda^p+\mu^p)(\|g_{1}\|_{wL^{\Phi}}^{p}+|\|g_{2}\|_{wL^{\Phi}}^{p})}\\&\geq\frac{(2\bar{\beta}_{\Phi}-\varepsilon)^{p}}{2C_{\Phi_1}C_{\Phi_2}}.\end{aligned}$$
		
        We can conclude that
		$$L_{YJ}^{C,p}(wL^\Phi(\mathbb{R}^n))\geq\frac{2^{p-1}\bar{\beta}_\Phi^p}{C_{\Phi_1}C_{\Phi_2}}.$$	
	\end{proof}
	Next, following the proof of Theorem \ref{t3}, we shall establish a lower bound for the $p$-th skew von Neumann-Jordan constant in weak Lebesgue spaces.
	\begin{Theorem}
			If $p>1$, then
		$$L_{YJ}^{C,p}(wL^{\Phi}(\mathbb{R}^{n}))\geq\max\left\{\frac{(\lambda+\mu)^p+|\mu-\lambda|^p}{C_{p_1}C_{p_2}(\lambda^p+\mu^p)},\frac{2^{p-2}}{C_{p_1}C_{p_2}}\right\} .$$
	\end{Theorem}
	\begin{proof}
			We commence by demonstrating the lower bound within inequality. Define the function $\Phi(u)$ as $\Phi(u)=u^p$. Based on Theorem \ref{t4}, we can obtain
		$\bar{\alpha}_{\Phi}=2^{-\frac{1}{p}}$ and $\bar{\beta}_{\Phi}=2^{-\frac{1}{p}}.$
		In the case where $\Phi=u^p$, it follows that $wL^\Phi(\mathbb{R}^n)=wL^p(\mathbb{R}^n)$. According to Theorem \ref{t3}, we then obtain
		$$L_{YJ}^{C,p}(wL^{\Phi}(\mathbb{R}^{n}))\geq\max\left\{\frac{(\lambda+\mu)^p+|\mu-\lambda|^p}{C_{p_1}C_{p_2}(\lambda^p+\mu^p)},\frac{2^{p-2}}{C_{p_1}C_{p_2}}\right\} .$$	
	\end{proof}
	
	\section*{Acknowledgments}
	Thanks to all the members of the Functional Analysis Research team of the College of Mathematics and Physics of Anqing Normal University for their discussion and correction of the diﬃculties and errors encountered in this paper.


\begin{thebibliography}{00}
		\bibitem{1} A.A. Masta, H. Gunawan, and W. Setya-Budhi, Inclusion property of Orlicz and weak
		Orlicz spaces, J. Math. Fund. Sci., 48-3 (2016), 193–203.
		\bibitem{2}J. A. Clarkson, The von Neumann-Jordan constant for the Lebesgue space, Ann. of
		Math. 38(1937), 114-115.
		\bibitem{3}P. Jordan, J. von Neumann, On Inner Product in Linear Metric Spaces, Ann. Math
		(2) 36(1935), 719-723.
		\bibitem{4}W. Orlicz, Linear Functional Analysis (Series in Real Analysis Volume 4), World
		Scientific, 1992, Singapore.
		\bibitem{5}M.M. Rao, Z.D. Ren, Applications of Orlicz Spaces, Marcel Dekker. Inc., 2002, New
		York.
		\bibitem{6}Liu, Q., Li, Y.: On a new geometric constant related to the Euler-Lagrange type identity in Banach spaces.
		Math. 9(2), 116–128 (2021)
		\bibitem{7}Liu, Q., Zhou, C., Sarfraz, M., Li, Y.: On new moduli related to the generalization of the parallelogram
		law. Bull. Malays. Math. Sci. Soc. 45, 307–321 (2022)
		\bibitem{8}Yang, X., Li, Y., Yang, C.: On the $L_{YJ}(\lambda,\mu,X)$ constant for the regular octagon space. Filomat. 38(5),
		1583–1593 (2024)
	    \bibitem{9}Ifronika,Gunawan, H.  Von Neumann constant for weak Orlicz spaces and weak Lebesgue spaces. AIP Conference Proceedings, 3029(1). (2024).
	    \bibitem{10}Liu N, Ye Y G. Weak Orlicz space and its convergence theorems. Acta Mathematica Scientia, 2010, 30B(5): 1492 - 1500.
	    \bibitem{11}Liu, P. D.,  Wang, M. F. Weak Orlicz spaces: Some basic properties and their applications to harmonic analysis. Science China Mathematics, 56(4), 789-802.(2013)
	    \bibitem{12}Cwikel M. The dual of weak Lp. Ann Inst Fourier Grenoble, 1975, 25: 85–126
	    \bibitem{13}Liu P D, Hou Y L, Wang M F. Weak Orlicz space and its applications to martingale theory. Sci China Math, 2010,53: 905–916
	    \bibitem{14}Lotz H P. Rearrangement invariant continuous linear fuctionals on weak L1. Positivity, 2008, 12: 119–132
	    \bibitem{15}Nakai E. On generalized fractional integrals on weak Orlicz spaces, BMO, the Morrey spaces and the Campanato
	    spaces. Studia Math, 2008, 188: 193–221
	    \bibitem{16}Soria J. Lorentz spaces of weak-type. Quart J Math, Oxford Ser, 1998, 49: 93–103
	    \bibitem{17}Weisz F. Weak martingale Hardy spaces. Prob Math Stat, 1998, 18: 133–148
	    \bibitem{18}Alonso, J., Martín, P., Papini, P. L. Wheeling around von Neumann-Jordan constant in Banach spaces. Studia Mathematica, 188(2), 135 - 150(2008).
	    \bibitem{19}Kato, M., Maligranda, L.,Takahashi, Y. (2001). On James and Jordan-von Neumann constants and the normal structure coefficient of Banach spaces. Studia Mathematica, 144(3), 275 - 290.
	    \bibitem{20}Jim\'{e}nez - Melado, A., Llorens - Fuster, E., Saejung, S.  The von Neumann - Jordan constant, weak orthogonality and normal structure in Banach spaces. Proceedings of the American Mathematical Society, 133(2), 353 - 364.(2005).
	    \bibitem{21}Yang, C. S., Li, H. Y.  Generalized von Neumann-Jordan constant for the Banaś-Fraçzek space. Colloquium Mathematicum, 154(1), 149 - 160.(2018).
	    \bibitem{22}Medina Galego, E.  The strongest forms of Banach-Stone theorem to \(C_0(K, \ell_p^n)\) spaces for all \(n \geq 3\) and $p$ close to 2. Journal of Mathematical Analysis and Applications, 541, 128715.(2025).
	    \bibitem{23}Xu, H. K. Iterative algorithms for nonlinear operators. Journal of the London Mathematical Society, 66(2), 240 - 256.(2002)
	    \bibitem{24}Kim, S. K., Lee, H. J. Uniform convexity and the Bishop-Phelps-Bollobás property. Canadian Journal of Mathematics, 66(2), 373 - 386. (2014). 
	    \bibitem{25}Medina Galego, E., Porto da Silva, A. L.  Isomorphisms of \(C_0(K, X)\) spaces with large distortion. Mathematische Nachrichten, 292(22), 996 - 1007.(2019).
	    \bibitem{26}Niculescu, C. P. Functional Inequalities in the Framework of Banach Spaces. Results in Mathematics, 79, 275.(2024).
	    \bibitem{27}Albiac, F.  Nonlinear structure of some classical quasi-Banach spaces and F-spaces. Journal of Mathematical Analysis and Applications, 340(2), 1312 - 1325.(2008). 
	    \bibitem{28}Yang X.R., Yang, C.S. An equality between $\mathrm{L}_{Y J}(\lambda, \mu, X)$ constant and generalized James constant, Math. Inequal. Appl., 27, 3 (2024), 571581.
	    \bibitem{29}N.J. Kalton, N.T. Peck, J.W. Rogers, An F-Space Sampler, London Math. Lecture Notes, vol. 89, Cambridge Univ. Press, Cambridge, 1985.
	    \bibitem{30}Albiac, F. Nonlinear structure of some classical quasi-Banach spaces and F-spaces. Journal of Mathematical Analysis and Applications, 340(2), 1312 - 1325. (2007).
	    \bibitem{31}Silva, E.B., Fernandez, D.L., Nikolova, L.: Generalized quasi-Banach sequence
	    spaces and measures of noncompactness. An. Acad. Bras. Ciˆ e. 85(2), 443–456 (2013)
	    \bibitem{32}\'{S}anchez, F.C., Garbul\'{i}nska-W\c{e}grzyn, J., Kub\'{i}s, W.: Quasi-Banach spaces of almost universal disposition. J. Funct. Anal. 267, 744–771 (2014)
	    \bibitem{33}Kalton, N.J.: Quasi-Banach Spaces, Handbook of the Geometry of Banach spaces, pp. 1101–1139. Elsevier, Amsterdam (2003)
	    \bibitem{34}Wang, Y. X., Liu, Q., Li, Q.,Huang, Q.  Some aspects of skew generalized Zbăganu constant in Banach spaces. Journal of Fixed Point Theory and Applications, 27(22).(2025).
	    \bibitem{35}Wang, Y. X., Liu, Q., Xia, J. Y., Huang, S. Z.  The skew generalized von Neumann–Jordan constant in the unit sphere. Journal of Mathematical Inequalities, 19(2), 561-579.(2025).
	    \bibitem{36}Ni, Q., Liu, Q., Zhou, Y. Skew generalized von Neumann-Jordan constant in Banach spaces. Hacettepe Journal of Mathematics and Statistics, 54(2), 136-144. (2025).
	    
	    
	\end{thebibliography}
\end{document}